\documentclass[a4paper]{article}
\usepackage{amsmath,amssymb,amsthm,amsfonts,graphicx,fullpage}
\usepackage{color}
\usepackage[applemac]{inputenc}

\def\Rset{\mathbb{R}}
\def\Nset{\mathbb{N}}

\def\bfo{{\mathbf 1}}

\def\ta{\alpha^\star}
\def\tb{\beta^\star}

\def\cN{{\cal N}}

\def\cN{{\cal G}}

\def\cN{{\cal N}}

\newcommand{\comment}[1]{}
\renewcommand{\t}{^{\mbox{\tiny\sf T}}}

\newtheorem{theo}{Theorem}
\newtheorem*{theo*}{Theorem}
\newtheorem{defi}{Definition}

\newtheorem{assu}{Assumption}
\newtheorem*{prop*}{Proposition}
\newtheorem{prop}[theo]{Proposition}

\theoremstyle{definition}
\newtheorem{rema}{Remark}

\begin{document}
\title{Tight estimates for convergence of some non-stationary consensus algorithms}
\author{David Angeli\thanks{Dipartimento di Sistemi e Informatica, 
University of Florence, Via di S.\ Marta 3, 50139 Firenze, Italy and INRIA, Rocquencourt BP105, 78153 Le Chesnay cedex, France.
Email: {\tt angeli@dsi.unifi.it}} \and Pierre-Alexandre Bliman\thanks{INRIA, 
Rocquencourt BP105, 78153 Le Chesnay cedex, France.
Email: \tt 
pierre-alexandre.bliman@inria.fr}}\maketitle


\begin{abstract}
The present paper is devoted to estimating the speed of convergence towards consensus for a general class of discrete-time multi-agent systems.
In the systems considered here, both the topology of the interconnection graph and the weight of the arcs are allowed to vary as a function of time.
Under the hypothesis that some spanning tree structure is preserved along time, and that some nonzero minimal weight of the information transfer along this tree is guaranteed, an estimate of the contraction rate is given.
The latter is expressed explicitly as the spectral radius of some matrix depending upon the tree depth and the lower bounds on the weights.
\vspace{.2cm}

\end{abstract}

\section{Introduction}

Appeared in the areas of communication networks, control theory and parallel computation, the analytical study of ways for reaching consensus in a population of agents is a problem of broad
interest in many fields of science and technology; see \cite{ANG07} for references.
Of particular interest is the question of estimating how quickly consensus is
reached on the basis of few qualitative (mainly topological) information as well as basic quantitative information on the network (mainly the strength of reciprocal influences).
    
Originally, this problem was considered in the context of
stationary networks. For Markov chains that are homogeneous (that is stationary in the vocabulary of dynamical systems), it amounts to quantify the speed at which 
steady-state probability distribution is achieved, and is therefore directly related to finding an
a priori estimate to the second largest eigenvalue of a stochastic matrix.
Classical works on this subject are due to Cheeger and Diaconis, \cite{cheeger,diaconis}, see also
\cite{friedland} for improved bounds, as well as \cite{SIN,KAH} and \cite{ROS} for a survey.
The latter concern reversible Markov chains, for example when the transition matrix is symmetric, see e.g.\ \cite{FIL} for the non-reversible case.

Among the classical contributions which instead deal with time-varying interactions we refer to the work of Cohn, \cite{cohn}, where asymptotic convergence is proved, but neglecting the issue of relating topology and
guaranteed convergence rates.
Tsitsiklis {\em et al.} also provided important qualitative contributions to this subject \cite{TSI84,TSI86,BER}, as well as Moreau \cite{MOR04a}.
See also \cite{ANG} for further nonlinear results.
In particular, the role of connectivity of the communication graph in the convergence of consensus and spanning trees has been recognised and finely analysed \cite{MOR04a,CAO05,OLS}.

More recently, important contributions in characterizing convergence to  consensus in a time-varying set-up were proven by several authors, see for instance \cite{BER, MOR04a}.
See also \cite{CAO05,BOY} for more specific cases.


In a previous paper \cite{ANG07}, several criteria were provided to estimate quantitatively the contraction rate of a set of agents towards consensus, in a discrete time framework.
The attempt there consisted in following the spread of the information over the agent population, along one or more spanning-trees. Ensuring a lower bound to the matrix entries of the agents already attained by the information flow along the spanning-tree, rather than the nonzero contributions as classically, permitted to obtain tighter estimates with weaker assumptions. 
Distinguishing between different sub-populations, of agents already touched by spanning-tree and agents not yet attained, and using lower bounds on the influence of the former ones, one is able to establish rather precise convergence estimates.

As a matter of fact, rapid consensus can be obtained in two quite different ways --- either by dense and isotropic communications (based, say, on a complete graph), or by very unsymmetric and sparse relations (with a star-shaped graph with a leading root).
In the first case many spanning trees cover the graph, while in the second configuration a unique one does the job.

The present article is a continuation of \cite{ANG07}.
Emphasis is put on propagation of a unique spanning tree and on the resulting consequences in terms of convergence speed.
It is demonstrated that in the particular case where such a spanning tree structure is guaranteed to exist at any time, ensuring minimal weight to the transmission of information along the tree (from the root to the leafs) indeed enforces some minimal convergence rate, whose expression is particularly simple.
A worst-case estimate is provided, expressed as the spectral radius of certain matrix whose size equals the depth of the tree and whose coefficients depend in a simple way of the assumed minimal weights.
This results in a sensible improvement over existing evaluations.

The paper is organized as follows.
Section \ref{se2} contains the problem formulation and a presentation of the main result, together with the minimal amount of technical tools to allow for its comprehension.
A comparison system is introduced afterwards in Section \ref{se3}, whose study is central to establish the convergence estimate.
The original method for analysis of this system is used in Section \ref{se4} to get convergence rate estimate (therein is stated the main result of the paper, Theorem \ref{th2}), and some properties of the latter are studied.
This result is commented in Section \ref{se6}, before some concluding remarks.
The proofs are sent back to Appendix.

\subsection*{Notations}

The $i$-th vector of the canonical basis in the space $\Rset^n$ ($1\leq i\leq n$) is denoted $e^n_i$; the vector with all components equal to 1 in $\Rset^n$ is written $\bfo^n$.
When the context is clear, we omit the exponent and just write $e_i$, resp.\ $\bfo$ to facilitate reading.
We also use brackets to select components of vectors.
All these notation are standard, and for a vector $x\in\Rset^n$, the $i$-th component is written alternatively $x_i$, $[x]_i$, $(e_i^n)\t x$ or $e_i\t x$.

The systems considered here will be composed of $n$ agents: accordingly, we let $\cN\doteq\{1,\dots , n\}$.


As usual, identity and zero square matrices of dimension $q \times q$ are denoted $I_q$ and $0_q$ respectively.
We denote $J_q$ the $q\times q$ matrix with ones on the sub-diagonal and zeros otherwise: $(J_q)_{i,j}=\delta_{i=j+1}$.
Here, and later in the text, $\delta$ denotes the Kronecker symbol, equal to 1 (resp.\ 0) when the condition written in the subscript is fulfilled (resp.\ is not).
For self-containedness, recall that a real square matrix $M$ is said stochastic (row-stochastic) if it is nonnegative with each row sum equal to 1.

The spectral radius of a square matrix $M$ is denoted $\lambda_{\max} (M)$.
Last, we use the notion of nonnegative matrices, meaning real matrices which are componentwise nonnegative.
Accordingly, the order relations $\leq$ and $\geq$ envisioned for matrices are meant componentwise.

\section{Problem formulation and presentation of the main result}
\label{se2}

Our aim is to estimate the speed of convergence towards consensus for the following class of time-varying linear systems:
\begin{equation}
\label{tv}
x (t+1) = A(t) x(t)
\end{equation}
where $A(t)\doteq (a_{i,j}(t))_{(i,j)} \in \Rset^{n \times n}$ is a sequence of stochastic matrices
(in particular, $A(t) \bfo = \bfo$; this is exactly the dual of what happens in the case of non-homogeneous Markov chains, where the probability distribution, written as a row vector $\pi(t)$, verifies rather a relation like $\pi(t+1)=\pi(t)A(t)$).

Let us first introduce some technical vocabulary to present in simplest terms the main result of the paper, afterwards enunciated in Section \ref{se4}.
The definition of the quantity we intend to estimate is as follows.

\begin{defi}[Contraction rate]
\label{de4}
We call {\em contraction rate\/} of system \eqref{tv} the number $\rho\in [0,1]$ defined as:
\[
\rho \doteq \sup_{x(0)}\
\limsup_{t\to +\infty} \left ( \frac{\displaystyle\max_{i\in\cN}x_i(t)-\min_{i\in\cN}x_i(t)}{\displaystyle\max_{i\in\cN}x_i(0)-\min_{i\in\cN}x_i(0)}  \right )^{\frac{1}{t}}\ ,
\]
where the supremum is taken on those $x(0)$ for which the denominator is nonzero.
\end{defi}
The contraction rate is thus related to the speed of convergence to zero of the agent set diameter.
In what follows, the latter plays the role of a Lyapunov function to study convergence to agreement.
For stationary systems, as is well known, the number $\rho$ is indeed the second largest eigenvalue of the matrix $A$. More in general, it corresponds to the second largest Lyapunov exponent of the considered sequence of matrices $A(t)$.\\

\begin{defi}[Communication graph]
\label{de2}
We call {\em communication graph} of system \eqref{tv} {\em at time $t$} the directed graph defined by the ordered pairs $(j,i)\in\cN\times\cN$ such that $a_{i,j}(t)>0$.
\end{defi}
In the present context, we use indifferently the terms ``node" or ``agent".

We now introduce assumptions on the existence of a constant hierarchical structure embedded in the communication graph, and on minimal weights attached to the corresponding links.

\begin{assu}
\label{as1}
For a given positive integer $T_d>0$, called the {\em depth of the communication graph}, assume the existence of nested sets $\cN_0,\dots,\cN_{T_d}$ such that
\begin{itemize}
\item $\cN_0$ is a singleton (whose element is called the {\em root});
\item $\cN_k\subset\cN_{k+1}$;
\item $\cN_{T_d}=\cN=\{1,\dots, n\}$.
\end{itemize}

Assume in addition, for given nonnegative real numbers $\alpha,\beta,\gamma$, that, for all $t\geq 0$ and all $k \in \{1,2, \ldots, T_d \}$
\begin{subequations}
\label{ineq}
\begin{align}
\label{ineq1}
a_{i,i}(t) & \geq \alpha \quad \text{ if } i \in \mathcal{N}_0,\\
\label{ineq2}
\sum_{j\in\cN_{k}\setminus\cN_{k-1}} a_{i,j}(t)&
\geq \beta \quad\text{ if } i \in \mathcal{N}_k \backslash \mathcal{N}_{k-1},\\
\label{ineq3}
\sum_{j\in\cN_{k-1}} a_{i,j}(t) & \geq \gamma\ \text{ if } i \in \mathcal{N}_k \backslash \mathcal{N}_{k-1}.
\end{align}
\end{subequations}
\end{assu}

As an example, the sets $\cN_k$ may be induced by some fixed spanning tree embedded in the communication graph: the existence of a distinguished agent, the root, is presupposed and, although the matrices $A(t)$ and the underlying communication graphs are allowed some variations, information progress from this root along a (time-varying) tree to attain all the agents.
The number $T_d$ bounds from above the minimal time for the information to attain the most distant agents from the root.
Likely, we call $d_i\doteq\min\{ k\ :\ i\in\cN_k\}$ the {\em depth of agent $i$}.
The set $\cN_k$ indeed consists of all the agents $i$ whose depth $d_i$ is guaranteed by Assumption \ref{as1} to be at most equal to $k$.
An example of (fixed) communication graph and the associated nested sets is shown in Figure \ref{treewave}.
 \begin{figure}
 \centerline{
 \includegraphics[width=8cm]{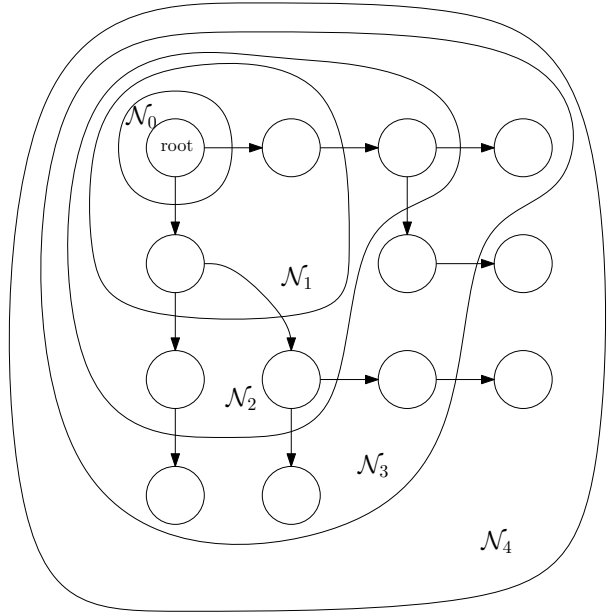} }
 \caption{The nested sets and the spanning tree}
 \label{treewave}
 \end{figure}

In addition to the spanning tree structure, Assumption \ref{as1} imposes some minimal weights to the information transmitted downstream along this structure (this is the role played by $\gamma$), and also to the information used between agents located at same depth.
Concerning the latter, expressed by condition \eqref{ineq2}, remark that it is fulfilled by self-loops, that is when
\[
a_{i,i}(t)\geq \beta \quad \text{ if } d_i>0
\]
(because by definition, $i\in\cN_{d_i}\setminus\cN_{d_i-1}$ for $d_i>0$); but it is indeed weaker: it allows just as well communications between agents whose depths are equal.
The constraint on the self-loops of the root agent, measured by $\alpha$, is different than for the other agents ($\beta$);
this is done on purpose, and permits to treat simultaneously the case of leaderless coordination and `pure' coordination with a leader 
(case corresponding to $\alpha = 1$).

Last, notice that, the matrices $A(t)$ being stochastic, one should have:
\[
\alpha,\ \beta+\gamma\leq 1
\]
for Assumption \ref{as1} to be fulfilled.\\

We are now in position to present the contents of Theorem \ref{th2}.
The latter states that, under the conditions exposed above, {\em the rate of convergence of system \eqref{tv} is {\em at most\/} equal to the spectral radius of the $T_d\times T_d$ matrix $\zeta_{T_d}(\alpha,\beta,\gamma)$ defined by}
\[
\zeta_{T_d}(\alpha,\beta,\gamma)
=
\begin{pmatrix}
\tb & 0 &  \dots & 0 & 1-\ta -\tb\\
\ta & \tb & \ddots & \vdots & 1-\ta-\tb\\
0 & \ta & \ddots & 0 & \vdots\\
\vdots & \ddots & \ddots & \tb & 1-\ta-\tb\\
0 & \dots & 0 & \ta & 1-\ta
\end{pmatrix},\quad
\text{ \em with } \ta=\min\{\alpha,\gamma\},\ \tb=\min\{\beta+\gamma,\alpha\}-\ta\ .
\]
A major characteristic of this estimate is that it is {\em independent\/} of the number $n$ of agents: it only depends upon $\alpha, \beta, \gamma$ and the depth $T_d$.\\

\comment{
Otherwise said, if the root uses at each time its own information with a weight at least $\alpha$; if moreover that at each time, every agents is at a `distance' at most $T_d$ of the root, and use with a weight at least $\gamma$, information received  from an agent located upstream, and its own information with a weight at least $\beta$, then the contraction rate is at most $\lambda_{\max}(\zeta_{T_d}(\alpha,\beta,\gamma))$.
}

We introduce, in the rest of the present Section, a general example where Assumption \ref{as1} is naturally fulfilled.

\begin{defi}[$(\alpha,\beta,\gamma)$-tree matrix]
For every nonnegative numbers $\alpha, \beta,\gamma$, we call {\em $(\alpha,\beta,\gamma)$-tree matrix} any matrix $M\in\Rset^{n\times n}$ defined by the recursion formula
\begin{equation*}
\label{ourmatrices}
M_1 = [\alpha],\quad
M_i = \left( \begin{array}{c|c} M_{i-1} & 0_{i-1} \\
\hline
\gamma u^{i-1} & \beta \end{array}
\right),\quad
M\doteq M_n\ ,
\end{equation*}
where, for all $i=2,\dots, n$, the vector $u^{i-1}$ is a vector of the canonical basis in $\Rset^{i-1}$.
\end{defi}

Notice that the agents have implicitly been numbered by the tree matrix representation: the pertinent information propagates from smaller to higher indexes.
A central case where Assumption \ref{as1} holds is given by the following result.

\begin{prop}
\label{pr1}
Let $A(t)$ be stochastic matrices.
Assume the existence of a sequence of $(\alpha,\beta,\gamma)$-tree matrices $M(t)\doteq (m_{i,j}(t))_{(i,j)} \in \Rset^{n \times n}$ such that
for all $t \in \Nset$
\[
\ A(t)\geq M(t)\ .
\]
Then, after some finite time, system \eqref{tv} fulfills Assumption {\rm\ref{as1}} with
\[
\cN_k = \{ i \in \cN : r_i \leq k \}
\]
and $T_d=r_n$ recursively defined as
\begin{equation*}
\label{ri}
r_1=0,\quad
r_i=1+\max\left\{
r_j\ :\ j<i \text{ and } m_{i,j}(t)>0 \text{ infinitely many times}
\right\},\ 2\leq i\leq n\ .
\end{equation*}
\end{prop}
We emphasize the fact that the lower bound $M(t)$ may vary upon time.

By construction, $T_d$ in Proposition \ref{pr1} verifies: $1\leq T_d\leq n-1$, and does not depend upon the ordering of the matrices $M(t)$.
Moreover, it may be proved directly that in the particular case of constant $M$, $T_d$ is the depth of the associated graph (defined in Definition \ref{de2}); generally speaking, however, the depth of a tree matrix sequence is {\em at least\/} equal to the $\limsup$ of the depths of the individual matrices $M(t)$.
To prove both properties, it is sufficient to remark that, in the case of constant $M$, the previous formula indeed computes the depth of the associated graph.
Figure \ref{switchingtree} presents the case of two matrices for which the supremum of depths is equal to 2, that is strictly less than the depth of the sequence of matrices obtained by alternatively taking each of them, which is here equal to 3 (and also strictly less, in this case, than $n-1=4$).
One can take the numbers $d_i$ defined in Assumption \ref{as1} equal to the corresponding numbers $r_i$ given below, a quite natural choice which yields in the present case:
\[
d_1=0,\ d_2=1,\ d_3= 1+\max\{d_1,d_2\}=1+\max\{0,1\}=2,\
d_4=d_5=3\ .
\]

\begin{proof}[Proof of Proposition \protect\rm\ref{pr1}]
Let the family of sets $\mathcal{N}_k$ be defined as in the statement.
Clearly $\mathcal{N}_0 = \{ 1 \}$ is a singleton, and also $m_{1,1} (t) \geq \alpha$ for all $t \geq 0$ as desired.
Moreover, $\mathcal{N}_k \backslash \mathcal{N}_{k-1} = \{ i \in \mathcal{N}: r_i=k \}$. In particular then, $i \in \mathcal{N}_{r_i} \backslash \mathcal{N}_{r_i-1}$ for all $i>1$.
Since $m_{i,i} (t) \geq \beta$ for all $i > 1$ and all $t \geq 0$, it is straightforward to verify that
\[
\sum_{j \in \mathcal{N}_{r_i} \backslash \mathcal{N}_{r_i-1} } m_{i,j} (t) \geq m_{i,i} (t) \geq \beta \]
as desired, for Assumption $1$ to hold. 
Finally, let $i \in \mathcal{N}_k \backslash \mathcal{N}_{k-1}$ , viz. $r_i = k$. This yields $m_{i,j} (t) > 0$ for some $j \in \mathcal{N}_{k-1}$
for infinitely many times. Actually, more is true due to the special structure of tree matrices, namely $m_{i,j}(t) \geq \gamma$ for infinitely many $t$s. We claim that for all $t$ larger than some finite time $T_i$ there exists $j(t) \in \mathcal{N}_{k-1}$ such that $m_{i,j(t)} (t) \geq \gamma$.
Indeed, let $F_t(i) \in \{1,2, \ldots i-1 \}$ denote the {\em father\/} of the $i$-th node in the tree matrix $M(t)$, that is the unique index $j$ such that $M_{i,j}(t)>0$; clearly,
$m_{i,F_t(i)} (t) \geq \gamma$ for all $t$.
Indeed, for all sufficiently large $t$s, $F_t (i) \in \mathcal{N}_{k-1}$, (otherwise $F_t (i) \notin  \mathcal{N}_{k-1}$
infinitely many times and, therefore, we would have
$r_i \geq k +1$, which is a contradiction).
Let $T_d= \max_{i  \in \mathcal{N} } T_i$.
For all subsequent times, we have:
\[ \sum_{j \in \mathcal{N}_{k-1} } m_{i,j}(t) \geq m_{i,F_t(i)} (t) = \gamma. \] 
This concludes the proof of the Proposition.
\end{proof}

\begin{figure}
 \centerline{
 \includegraphics[width=4cm]{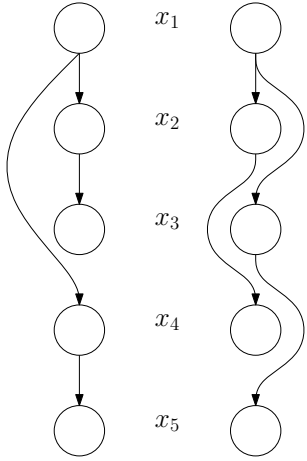} }
 \caption{Trees of depth $2$ inducing a nested structure of depth $3$.}
 \label{switchingtree}
 \end{figure}
 
\section{A comparison system for the diameters evolution}
\label{se3}

We now build an auxiliary time-varying system, with a simpler structure than \eqref{tv}, and with the property that the asymptotic contraction rate of the original system can be bounded from above by carrying out suitable computations on this newly introduced system.
Our main result for the present section is a statement relating convergence of \eqref{tv} towards consensus of a comparison system introduced below.

\begin{theo} 
\label{listofmatrix}
Assume system \eqref{tv} fulfills Assumption {\rm\ref{as1}}, for given nonnegative numbers $\alpha,\beta,\gamma$ (such that $\alpha,\beta+\gamma\leq 1$).
Let $\Delta(t)$ be defined by
\[
\Delta (t)\doteq
\begin{pmatrix}
\displaystyle\max_{i \in \cN_0} x_i(t)-\min_{i \in \cN_0} x_i(t)\\
\displaystyle\max_{i \in \cN_1} x_i(t)-\min_{i \in \cN_1} x_i(t)\\
\vdots \\
\displaystyle\max_{i \in \cN_{T_d}} x_i(t)-\min_{i \in \cN_{T_d}} x_i(t)
\end{pmatrix}\ .
\]

Then, $\Delta(t)$ satisfies the following inequality:
\begin{equation}
\label{ineqdiameter}
\Delta(t+1) \leq  \left(
\begin{array}{c|c}
1 & 0_{T_d\times 1}\\
\hline
\begin{matrix}
\ta \\ 0_{(T_d-1)\times 1}
\end{matrix} &
\zeta_{T_d} (\ta, \tb)
\end{array}
\right)\Delta(t)\ ,
\end{equation}
where $\zeta_{T_d} ( \ta,\tb)\in\Rset^{T_d\times T_d}$
\begin{gather}
\label{zeta}
 \zeta_{T_d} ( \ta,\tb) \doteq (1-\ta-\tb)\ \bfo e\t_{T_d} +\tb I_{T_d}+\ta J_{T_d},
\\
\label{tatb}
\ta\doteq\min\{\alpha,\gamma\},\quad \tb\doteq\min\{\beta+\gamma,\alpha\}-\ta\ .
\end{gather}
\end{theo}

Recall that inequality \eqref{ineqdiameter} is meant componentwise.
A complete proof of Theorem \ref{listofmatrix} is provided in Section \ref{pr_list}.

\begin{rema}
Two special cases of interest as far as application of Theorem \ref{listofmatrix} are obtained for the following values of parameters:
\begin{enumerate}
\item $\alpha=1$: viz. communication graph admits a {\em leader}; under such premises, expressions for $\ta$ and $\tb$ simplify as follows:
\[ \ta= \gamma \qquad \tb= \beta \]
\item $\alpha = \beta$, viz. root agent is not different from any other member of the group in terms of self-confidence on his own position
in the formation of consensus:
\[  \ta = \min \{ \beta, \gamma \} \qquad \tb = \max \{0, \beta - \gamma \} \]
\end{enumerate}
\end{rema}

\section{Convergence rate estimate and properties}
\label{se4}

Based on Theorem \ref{listofmatrix}, we now provide Theorem \ref{th2}, which states properly the property announced in the beginning of the paper.

\begin{theo}
\label{th2}
Consider the linear time-varying dynamical system \eqref{tv}, with $A(t)$ stochastic.
Assume Assumption {\rm\ref{as1}} is fulfilled.
Then, the contraction rate towards consensus can be bounded according to the
following formula:
\begin{equation}
\label{spectralbound}
\rho \leq \rho_{T_d} (\ta,\tb)\doteq \lambda_{\max} (\zeta_{T_d}(\ta,\tb))  \ ,
\end{equation}
with $\zeta_{T_d}, \ta, \tb$ given in \eqref{zeta} and \eqref{tatb}.
\end{theo}

Proof of Theorem \ref{th2} is given in Section \ref{pr_th2}.
Recall that stochasticity of $A(t)$ implies that the nonnegative scalar $\alpha, \beta, \gamma$ verify: $\alpha\leq 1$, $\beta+\gamma\leq 1$.

Theorem \ref{th2} provides a tight estimate for the contraction rate of \eqref{tv} on the basis of the parameters $\alpha$, $\beta$ and $\gamma$, and of the depth $T_d$ of the sequence of tree matrices.
We emphasize the fact that the result holds for time-varying systems.
Indeed, Theorem \ref{th2} is an inherently robust result, as Assumption \ref{as1} allows for much uncertainty in the definition of system \eqref{tv}.
This robustness is meant with respect to variations of the communication graph (provided these variations don't violate the set conditions of Assumption \ref{as1}), and with respect to variations of the coefficients of the matrix $A(t)$ (provided they respect the quantitative constraints in Assumption \ref{as1}).

A central fact is that the value in \eqref{spectralbound} does not depend upon the number of agents involved in the network: rather the depth of the graph is involved, which is quite natural.\\
 
 
Some properties of the estimate are now given.
They are indeed useful to have a grasp on the asymptotic behaviour of the contraction estimate, as well as on their monotonicity properties; the latter are in agreement with the increase of decrease of information available by varying the parameters $\alpha, \beta $ and $\gamma$. 

\begin{theo}
\label{th3}
Let $\ta,\tb\in (0,1]$.
Then for any $T\in\Nset$, $\rho_T(\ta,\tb)=\lambda_{\max} (\zeta_T(\ta,\tb))$ has the following properties.

\begin{itemize}
\item
$\rho_T(\ta,\tb)$ is the largest {\em real\/} root of the polynomial equation
\begin{equation}
\label{real}
\left(
\frac{s-\tb}{\ta}
\right)^T
+\left(
\frac{s-\tb}{\ta}
\right)^{T-1}
+\dots
+\frac{s-\tb}{\ta}
+1
=
\left(
\frac{1-\tb}{\ta}
\right)
\left(
\left(
\frac{s-\tb}{\ta}
\right)^{T-1}
+\dots +\frac{s-\tb}{\ta}+1
\right)\ .
\end{equation}
\item
For any $T \in\Nset$, $\rho_T(\ta,\tb)\leq \rho_{T+1}(\ta,\tb)$.
\item
For any $T \in\Nset$, $1-\ta,\tb< \rho_T(\ta,\tb)<1$.
\item
$ \rho_T(\ta,\tb)\leq\ta+\tb$ if and only if $T\leq\frac{\ta}{1-\ta-\tb}$.
\item
$\rho_T(\ta,\tb)\to 1$ when $T \to +\infty$, and more precisely
\[
\rho_T(\ta,\tb)= 1-(1-\ta-\tb)\left(
\frac{\ta}{1-\tb}
\right)^T
+ o\left(\left(
\frac{\ta}{1-\tb}
\right)^T\right)\ .
\]
\end{itemize}
\end{theo}

Theorem \ref{th3} is demonstrated in Section \ref{pr_th3}.

The following result, demonstrated in Section \ref{pr_th4}, studies the variation of $\rho_T$ as a function of $\alpha, \beta, \gamma$.
When considering $\rho_T$ as a function of these quantities, we write $\rho_T(\alpha, \beta, \gamma)$, meaning $\rho_T(\ta,\tb)$ for $\ta(\alpha,\beta,\gamma),\tb(\alpha,\beta,\gamma)$ defined as in \eqref{tatb}.

\begin{theo}
\label{th4}
For any $T\in\Nset$,
\begin{itemize}
\item
the function $(\ta,\tb)\mapsto\rho_T(\ta,\tb)$ is nonincreasing on the set $\{(\ta,\tb)\in [0,1]^2\ :\ \ta+\tb\leq 1\}$; 
\item
the function $(\alpha,\beta,\gamma)\mapsto\rho_T(\alpha,\beta,\gamma)$ is nonincreasing on the set $\{(\alpha,\beta,\gamma)\in [0,1]^3\ :\ \beta+\gamma\leq 1\}$;
\item
if $\beta+\gamma=\beta'+\gamma'$, then $\rho_T(\alpha,\beta,\gamma)\leq\rho_T(\alpha,\beta',\gamma')$ when $\beta\geq\beta'$.
\end{itemize}
Moreover, for any $T\in\Nset$,
\begin{itemize}
\item
$\rho_T(\alpha,\beta,\gamma)=1$ if and only if $\alpha=0$ or $\gamma=0$.
\item
$\rho_T(\alpha,\beta,\gamma)=\beta=1-\gamma$ if and only if $\alpha=\beta+\gamma=1$.
\end{itemize}
\end{theo}

 Notice that the estimates given in the last two points of Theorem \ref{th4} are tight: they are reached for the following {\em stationary\/} systems:
 
\begin{align*}
\text{Case $\alpha=0$:} & \quad A=J_n +e_1e_n\t\\
\text{Case $\gamma=0$:} & \quad A=I_n\\
\text{Case $\alpha=\beta+\gamma=1$:} & \quad A=\beta I_n+(1-\beta)(J_n+e_1e_1\t) 
\end{align*}

\section{Discussion and interpretation of the results}
\label{se6}


It is interesting to compare our results with the classical estimate $\rho \leq \sqrt[T_d]{ 1 - \alpha^{T_d} }$ which is obtained by assuming a lower-bound $\alpha$ on the diagonal entries as well as on the non-zero entries of $A(t)$. In our set-up this is obtained by letting $\alpha= \gamma =
\beta=\ta$ and $\tb=0$. 
In order to have an idea on the quality of the two estimates, we plot the ratio of the \emph{spectral gaps}, 
 \[ \frac{1 - \rho_{T_d }( \ta, 0 ) }{1 - \sqrt[T_d]{ 1 - {\ta}^{T_d} }} \]
for $T_d=2,3,4$ in Fig. \ref{betazero}. As it is possible to see, the new estimates are consistently tighter than the classic ones; in the best case,
viz. for $\ta \approx 0$, the ratio of spectral gaps approaches $T_d$. So, the quality of the estimates actually improves with respect to the classic bound, as the horizon $T_d$ increases. 
\begin{figure}
\centerline{
\includegraphics[width=14cm]{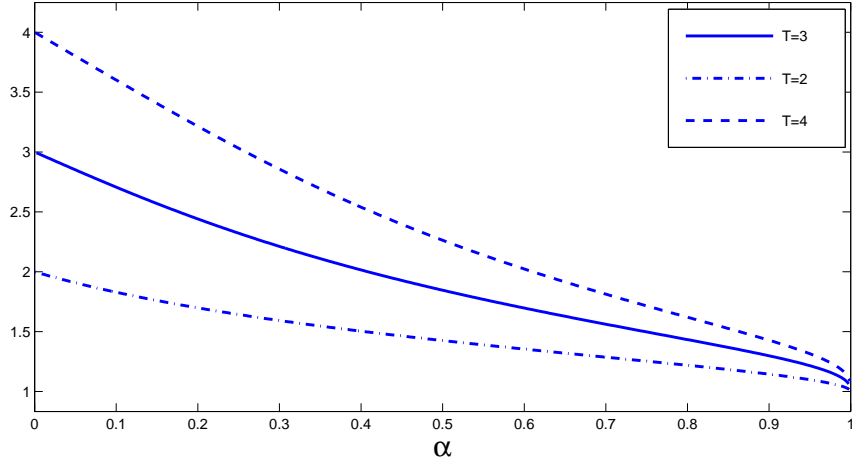}}
\caption{Ratios between spectral gaps}
\label{betazero}
\end{figure}

When additional information is available, for instance when the coefficient $\alpha, \beta, \gamma$ as given in (\ref{ineq}) are known, then contraction rate estimates become much tighter with respect to their classical counterparts which are not able to discriminate between inner loops of the root node and inner-loops of individual agents, as well as strength of inter-agent communication links.
In order to carry out a comparison, notice that under the assumption of a prescribed $\alpha, \beta, \gamma$ tree matrix bounding from below $A(t)$, we may assume for the classical estimate the following value of $\alpha := \min \{ \alpha, \beta, \gamma \}$ which indeed is always smaller than $\ta=\min \{ \alpha, \gamma \}$.
Hence, the corresponding spectral gaps satisfy:
\[
1 - \sqrt[T_d]{1 -{\min \{ \alpha, \beta, \gamma \} }^{T_d}} \leq 1- \sqrt[T_d]{1 - \min \{ \alpha, \gamma \}^{T_d} } 
\]
so that, we may compare the classical estimate with the new one by considering the following ratios:
\[ \frac{1 - \rho_{T_d} ( \ta, \tb )}{ 1 - \sqrt[T_d]{1 -{\min \{ \alpha, \beta, \gamma \} }^{T_d}} } \geq \frac{ 1- \rho_{T_d} (\ta, \tb) }{1-\sqrt[T_d]{1 - {\ta}^{T_d} } } \]

We plotted the function at the right-hand side of the previous inequality
 in a $\log_{10}$ scale as a function of $\ta$ and $\tb$. In general the ratio depends critically on the tree depth $T_d$, hence we only plot it for relatively small tree depths.
In particular the results shown in Fig. \ref{compa1} were obtained.
\begin{figure}
\centerline{
\includegraphics[width=10cm]{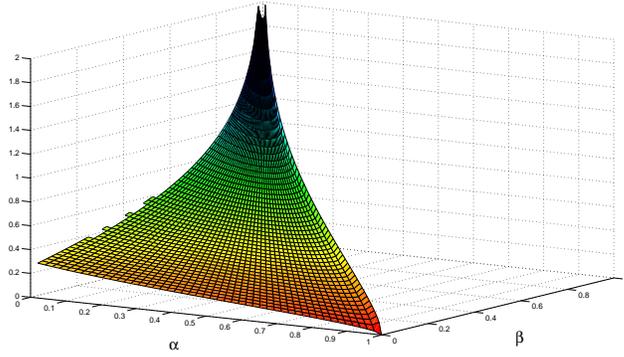} (a)}
\centerline{
\includegraphics[width=10cm]{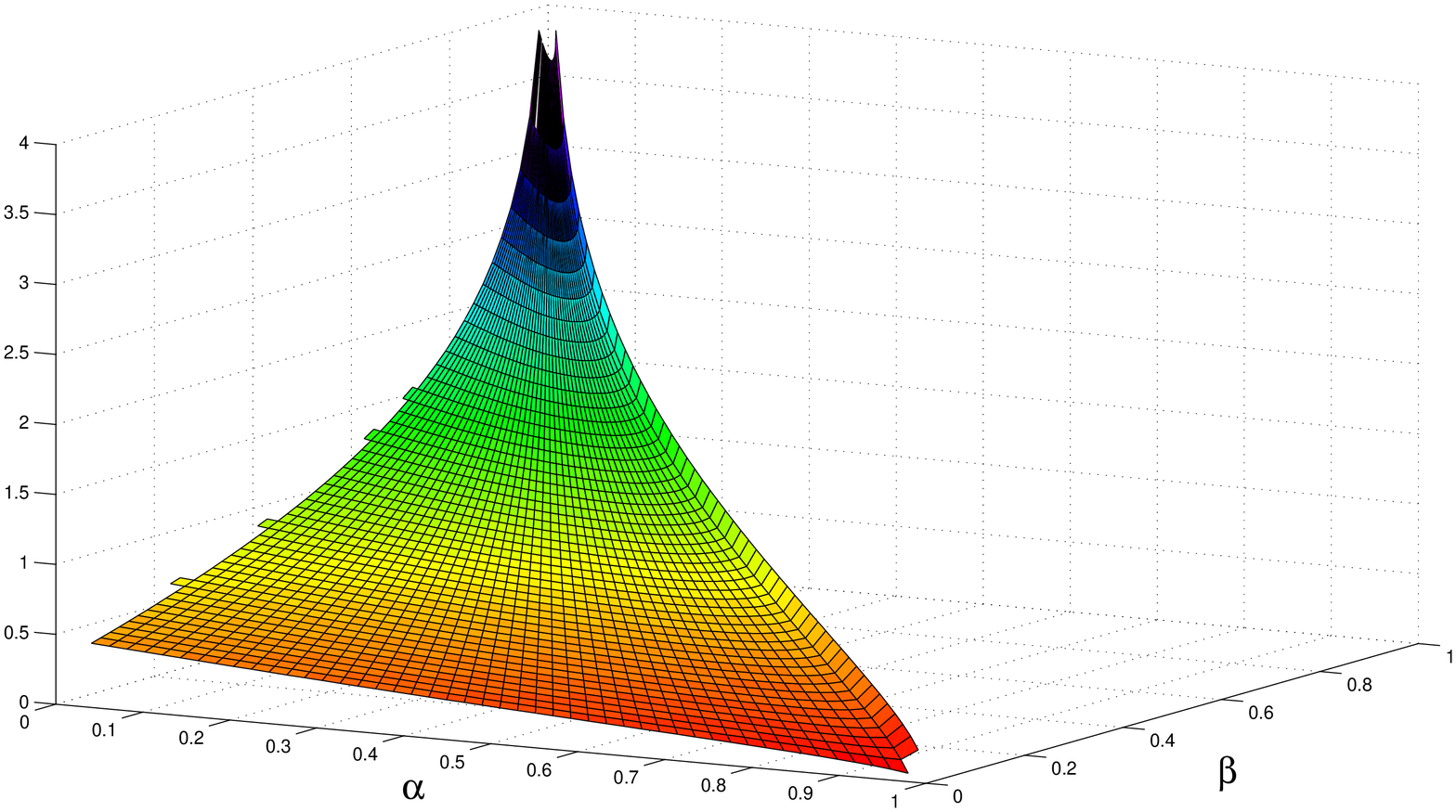} (b) } 
\centerline{
\includegraphics[width=10cm]{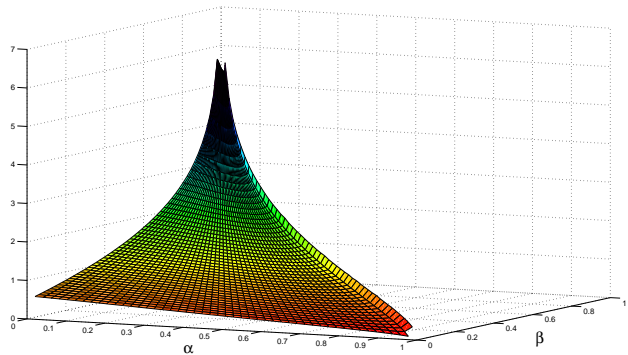} (c) }
\caption{Ratios of spectral gap: (a) $T=2$, (b) $T=3$, (c) $T=4$.
The vertical axis is graduated in a $\log_{10}$ scale.
}
\label{compa1}
\end{figure}
Notice that the relative quality of the estimates again increases with $T_d$, and already for $T_d=4$ a significant portion of parameters space lies in the area in which estimates differ by a $10^4$ factor. 
The dependence of $\rho_{T_d}$  upon $\ta$ and $\tb$ is shown in Fig.\ \ref{compa2} for $T=2,3,4$. This also clearly shows the different monotonicity properties highlighted in the previous Section.
\begin{figure}
\centerline{
\includegraphics[width=10cm]{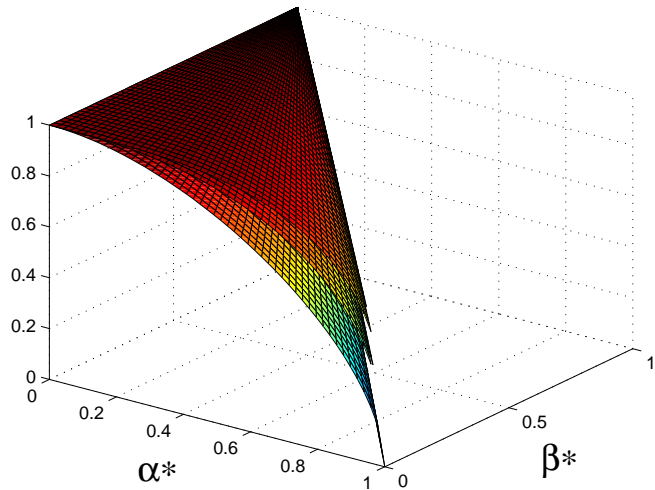} }
\caption{The function $\rho_{T_d} ( \ta, \tb)$ for $T_d=2,3,4$ (from bottom to top) }.
\label{compa2}
\end{figure}

\section{Conclusion}

We provide a novel and tight estimate of the contraction rate of infinite products of stochastic matrices, under the assumption of  prescribed lower bounds on the influence between different sets of agents which naturally arise by following the information spread along the interaction graph. 
This improves previously known bounds and, when additional information is assumed, exploits the additional structure for tightening of several orders of magnitude the previously available estimates. The other crucial factor in determining the overall convergence rate is the time $T_d$ needed to the information to propagate from some root node (which may or may not play the role of a leader) to the other nodes.
The bound can be computed as the Perron-Frobenius eigenvalue (the spectral radius) of a positive $T_d$-dimensional matrix, whose entries depend in a relatively simple way on the parameters characterizing the hypothetic lower bounds available. Some monotonicity and asymptotic properties of the bound are also proved.

\appendix
 
\section{Proofs}

\subsection{Proof of Theorem \protect\ref{listofmatrix}}
\label{pr_list}

For each given solution $x(t)$ of \eqref{tv} we define two vectors of size $T_d+1$ as follows:
\begin{equation} 
\label{comparison}
z(t)\doteq
\begin{pmatrix}
\displaystyle\max_{i \in \cN_0} x_i(t)\\
\displaystyle\max_{i \in \cN_1} x_i(t)\\
\vdots \\
\displaystyle\max_{i \in \cN_{T_d}} x_i(t)
\end{pmatrix},\quad
w(t)\doteq
\begin{pmatrix}
\displaystyle\min_{i \in \cN_0} x_i(t)\\
\displaystyle\min_{i \in \cN_1} x_i(t)\\
\vdots \\
\displaystyle\min_{i \in \cN_{T_d}} x_i(t)
\end{pmatrix}\ ,
\end{equation}
in such a way that
\[
z_k(t)= \max_{i\in\cN_{k-1}}x_i(t),\quad
w_k(t)= \min_{i\in\cN_{k-1}}x_i(t)\ .
\]
Notice that $\Delta(t)$ defined in Theorem \ref{listofmatrix} equals $z(t)-w(t)$, and that, for any $i\in\cN$, $w_{T_d+1}(t)\leq x_i(t)\leq z_{T_d+1}(t)$.

Our first aim is to compute an upper-bound of $z(t+1)$ on the basis of $z(t)$. 
Let us consider the following estimates.
First, for the unique index $i_0\in\cN_0$,
\begin{eqnarray}
\nonumber
x_{i_0} (t+1) &=& \sum_l a_{i_0,l} (t) x_l (t) = a_{i_0,i_0} (t) x_{i_0} (t) + \sum_{l\neq i_0} a_{i_0,l} (t) x_l (t) \\
\nonumber
 & \leq & a_{i_0,i_0} (t) x_{i_0} (t) + \left ( \sum_{l\neq i_0} a_{i_0,l} (t) \right ) z_{T_d+1}(t)
 = a_{i_0,i_0}(t) \, x_{i_0} (t) + ( 1- a_{i_0,i_0} (t) ) \, z_{T_d+1}(t) \\
\nonumber
& = & a_{i_0,i_0} (t) \, (x_{i_0}(t)- z_{T_d+1}(t) )   + z_{T_d+1}(t) 
\leq \alpha  \, x_{i_0} (t) + ( 1- \alpha ) \, z_{T_d+1}(t)\\
\label{x1}
& = &
\alpha  \, z_1 (t) + ( 1- \alpha ) \, z_{T_d+1}(t)\ ,
\end{eqnarray}
using the fact that $a_{i_0,i_0}(t)\geq\alpha\geq 0$ and $z_1(t)=x_{i_0}(t)\leq z_{T_d+1}(t)$.
Also, for $i\not\in\cN_0$ (that is $d_i>0$):
\begin{eqnarray}
\nonumber
x_i (t+1) &=& \sum_l a_{i,l} (t) x_l (t)
= \sum_{l\in\cN_{d_i}\setminus\cN_{d_i-1}} a_{i,l}(t) x_l (t) + \sum_{l\in\cN_{d_i-1}} a_{i,l}(t) x_l (t)+ \sum_{l\not\in\cN_{d_i}} a_{i,l}(t) x_l (t)\\
\nonumber
&\leq &
\sum_{l\in\cN_{d_i}\setminus\cN_{d_i-1}} a_{i,l}(t) x_l (t) + \sum_{l\in\cN_{d_i-1}} a_{i,l}(t) x_l (t)+ \left(
\sum_{l\not\in\cN_{d_i}} a_{i,l}(t)
\right) z_{T_d+1} (t)\\
\nonumber
&= &
\sum_{l\in\cN_{d_i}\setminus\cN_{d_i-1}} a_{i,l}(t) (x_l (t)-z_{T_d+1}(t)) + \sum_{l\in\cN_{d_i-1}} a_{i,l}(t) (x_l (t)-z_{T_d+1}(t))+ z_{T_d+1} (t)\\
\nonumber
& \leq &
\sum_{l\in\cN_{d_i}\setminus\cN_{d_i-1}} a_{i,l}(t) (z_{d_i+1}(t)-z_{T_d+1}(t)) + \sum_{l\in\cN_{d_i-1}} a_{i,l}(t) (z_{d_i}(t)-z_{T_d+1}(t))+ z_{T_d+1} (t)\\
\label{xp}
 & \leq & \beta \, z_{d_i+1}(t) + \gamma \, z_{d_i} (t) + (1 - \beta - \gamma ) \, z_{T_d+1}(t)
\end{eqnarray}
where the last inequality follows considering that $z_{d_i} (t)\leq z_{d_i+1}(t) \leq z_{T_d+1}(t)$ and $\beta \leq \sum_{l\in\cN_{d_i}\setminus\cN_{d_i-1}} a_{i,l}(t)$, $\gamma \leq \sum_{l\in\cN_{d_i-1}} a_{i,l}(t)$, by Assumption \ref{as1}.

We now proceed to compute suitable estimates for the vector $z$ defined in \eqref{comparison}.
In particular, for any $k\leq T_d$, we may derive, exploiting \eqref{x1} and \eqref{xp}:
\[
\max_{i\in\cN_k} x_i (t+1)
\leq \max \big \{ \alpha z_1 (t) + (1 - \alpha) z_{T_d+1} (t),\
\beta z_{k+1}(t) + \gamma z_k (t)+ (1 - \beta - \gamma)\ z_{T_d+1}(t) \big  \}\ .
\]
One has $z_1(t)\leq z_k(t)\leq z_{k+1}(t)\leq z_{T_d+1}(t)$.
Hence, the previous inequality implies:
\[
z_{k+1}(t+1)
\leq
\max \big \{ \alpha z_k (t) + (1 - \alpha) z_{T_d+1} (t),\
\beta z_{k+1}(t) + \gamma z_k (t)+ (1 - \beta - \gamma)\ z_{T_d+1}(t) \big  \}
\]
and thus
\begin{multline*}
z_{k+1}(t+1)
\leq
\min \big \{ \beta, \max \{ \alpha- \gamma , 0 \} \big \}\ z_{k+1}(t)  + \min \{ \alpha,\gamma \}\ z_k (t)\\
+ \big (1- \min \{\alpha, \gamma \} - \min \{ \beta, \max \{ \alpha- \gamma , 0 \} \big \} \big )\ z_{T_d+1}(t)\ ,
\end{multline*}
that is:
\begin{equation}
\label{comp2b}
z_{k+1}(t+1)
\leq
\tb\ z_{k+1}(t)  + \ta\ z_k (t)
+ \big (1- \ta-\tb \big )\ z_{T_d+1}(t)
\end{equation}
with the nomenclature adopted in \eqref{tatb}.

A symmetric argument can be carried out for the minima $w$ defined in \eqref{comparison}.
In analogy to the formulas \eqref{x1}, \eqref{xp} obtained in the previous
paragraphs, we get:
\begin{eqnarray}
\nonumber
x_{i_0} (t+1) &=& \sum_l a_{i_0,l} (t) x_l (t) = a_{i_0,i_0} (t) x_{i_0} (t) + \sum_{l\neq i_0} a_{i_0,l} (t) x_l (t) \\
\nonumber
 & \geq & a_{i_0,i_0} (t) x_{i_0} (t) + \left ( \sum_{l\neq i_0} a_{i_0,l} (t) \right ) w_{T_d+1}(t)
 = a_{i_0,i_0}(t) \, x_{i_0} (t) + ( 1- a_{i_0,i_0} (t) ) \, w_{T_d+1}(t) \\
& = &
\nonumber
a_{i_0,i_0} (t) \, (x_{i_0}(t)- w_{T_d+1}(t) )   + w_{T_d+1}(t) 
\geq \alpha  \, w_1 (t) + ( 1- \alpha ) \, w_{T_d+1}(t)\ ,
\end{eqnarray}
and, for $i\not\in\cN_0$,
\begin{eqnarray}
\nonumber
x_i (t+1) &=& \sum_l a_{i,l} (t) x_l (t)
= \sum_{l\in\cN_{d_i}\setminus\cN_{d_i-1}} a_{i,l}(t) x_l (t) + \sum_{l\in\cN_{d_i-1}} a_{i,l}(t) x_l (t)+ \sum_{l\not\in\cN_{d_i-1}\cup\{i\}} a_{i,l}(t) x_l (t)\\
\nonumber
&\geq &
\sum_{l\in\cN_{d_i}\setminus\cN_{d_i-1}} a_{i,l}(t) x_l (t) + \sum_{l\in\cN_{d_i-1}} a_{i,l}(t) x_l (t)+ \left(
\sum_{l\not\in\cN_{d_i-1}\cup\{i\}} a_{i,l}(t)
\right) w_{T_d+1} (t)\\
\nonumber
&= &
\sum_{l\in\cN_{d_i}\setminus\cN_{d_i-1}} a_{i,l}(t) (x_l (t)-w_{T_d+1}(t)) + \sum_{l\in\cN_{d_i-1}} a_{i,l}(t) (x_l (t)-w_{T_d+1}(t))+ w_{T_d+1} (t)\\
\nonumber
& \geq &
\sum_{l\in\cN_{d_i}\setminus\cN_{d_i-1}} a_{i,l}(t) (w_{d_i+1}(t)-w_{T_d+1}(t)) + \sum_{l\in\cN_{d_i-1}} a_{i,l}(t) (w_{d_i}(t)-w_{T_d+1}(t))+ w_{T_d+1} (t)\\
\nonumber
 & \geq & \beta \, w_{d_i+1}(t) + \gamma \, w_{d_i} (t) + (1 - \beta - \gamma ) \, w_{T_d+1}(t)\ .
 \end{eqnarray}

Similarly to \eqref{comp2b}, we get
\begin{equation*}
\label{comp2bbis}
w_{k+1}(t+1)
\geq
\tb\ w_{k+1}(t)  + \ta\ w_k (t)
+ \big (1- \ta-\tb \big )\ w_{T_d+1}(t)\ .
\end{equation*}

Putting now together \eqref{comp2b} and \eqref{comp2bbis} leads to:
\begin{equation}
\label{ineqD}
\Delta_{k+1}(t+1)
\leq
\tb\ \Delta_{k+1}(t)  + \ta\ \Delta_k (t)
+ \big (1- \ta-\tb \big )\ \Delta_{T_d+1}(t)\ .
\end{equation}
On the other hand, one also have
\[
\Delta_0(t+1)=0=\Delta_0(t)\ .
\]
Gathering this inequalities yields \eqref{ineqdiameter} and proves Theorem \ref{listofmatrix}.

\begin{rema}
\label{here}
Notice that alternatively to \eqref{comp2b}, the following estimate is also valid:
\begin{eqnarray*}
\nonumber
\max_{i\in\cN_k} x_i (t+1)
& \leq &
\max \big \{ \alpha z_1 (t) + (1 - \alpha) z_{T_d+1} (t),\
\beta z_{k+1}(t) + \gamma z_k (t)+ (1 - \beta - \gamma)\ z_{T_d+1}(t) \big  \}\\
& \leq &
\min  \{ \alpha, \beta \}\ z_{k+1}(t)  + \min \big \{ \gamma, \max \{ \alpha- \beta , 0 \} \big \}\ z_k (t)\\
& &
+ \big (1- \min \{\alpha, \beta \} - \min \{ \gamma, \max \{ \alpha- \beta , 0 \} \big \} \big )\ z_{T_d+1}(t)\ .
\end{eqnarray*}
This yields the result of Theorem \ref{listofmatrix} with
$\ta
= \min \{ \gamma, \max \{ \alpha - \beta, 0 \} \}$,
$\tb
= \min \{ \alpha, \beta \}$
instead of \eqref{tatb}, but the sequel demonstrates that the corresponding estimates are less precise (see Remark \ref{convexfreedom2} below).
\hfill$\square$
\end{rema}

\subsection{Proof of Theorem \protect\ref{th2}}
\label{pr_th2}

All the factors in \eqref{ineqdiameter} being nonnegative, the order relation is compatible with multiplication.
One then obtains, for all $t\in\Nset$,
\[
\Delta(t)
\leq  \left(
\begin{array}{c|c}
1 & 0_{T_d\times 1}\\
\hline
\begin{matrix}
\ta \\ 0_{(T_d-1)\times 1}
\end{matrix} &
\zeta_{T_d} (\ta, \tb)
\end{array}
\right)^t\Delta(0)
\leq  \left(
\begin{array}{c|c}
1 & 0_{T_d\times 1}\\
\hline
\begin{matrix}
\ta \\ 0_{(T_d-1)\times 1}
\end{matrix} &
\zeta_{T_d} (\ta, \tb)
\end{array}
\right)^t\begin{pmatrix}
0 \\ \bfo^{T_d}
\end{pmatrix}
\Delta_{T_d+1}(0)\ ,
\]
where the fact that $\Delta_1(t)\equiv 0\leq\Delta_k(t)\leq \Delta_{k+1}(t)\leq \Delta_{T_d+1}(t)$, $1\leq k\leq T_d+1$, have been taken into account.

One deduces that
\[
\Delta(t)
\leq \begin{pmatrix}
0_{T_d\times 1}\\ \zeta_{T_d} (\ta, \tb)^t
\end{pmatrix}
\bfo^{T_d}\Delta_{T_d+1}(0)
\]
and
\[
\Delta_{T_d+1}(t)
= e_{T_d+1}^{(T_d+1)\mbox{\tiny\sf T}}\Delta(t)
\leq e_{T_d}^{T_d\mbox{\tiny\sf T}}\zeta_{T_d} (\ta, \tb)^t \bfo^{T_d}\Delta_{T_d+1}(0)\ ,
\]
from which it ensues
\begin{multline*}
\limsup_{t\to +\infty}
\left(
\frac{\Delta_{T_d+1}(t)}{\Delta_{T_d+1}(0)}
\right)^{1/t}
\leq
\limsup_{t\to +\infty}
\left(
e_{T_d}^{T_d\mbox{\tiny\sf T}}\zeta_{T_d} (\ta, \tb)^t \bfo^{T_d}
\right)^{1/t}\\
\leq
\limsup_{t\to +\infty}
\left\|
\zeta_{T_d} (\ta, \tb)^t
\right\|^{1/t}
= \lambda_{\max} ( \zeta_{T_d} ( \ta ( \lambda ), \tb( \lambda ) ) )\ ,
\end{multline*}
see e.g.\ \cite[Corollary 5.6.14]{HOR90} for a demonstration of the previous equality.
This yields \eqref{spectralbound} and achieves the proof of Theorem \ref{th2}.

\begin{rema}
\label{convexfreedom2}
Recall that by virtue of Remark \ref{here}, one is allowed to choose $\ta$ and $\tb$ 
according to \eqref{tatb} or to
\[
\tb\doteq \min \{ \alpha, \beta \},\quad
\ta\doteq \min \{\beta+\gamma,\alpha\}-\tb\ .
\]
Besides, due to linearity of $C(t)$ with respect to $\ta$ and $\tb$, any convex combination of the two formulas can still be adopted as a suitable value of $(\ta,\tb)$ for the comparison system.
In particular then:
\begin{equation*}
\label{convexexpression}
\begin{array}{rcl}
\ta ( \lambda )&=& \lambda \min \{ \alpha, \gamma \} + (1- \lambda )  \min \{ \gamma, \max \{ \alpha - \beta, 0 \} \} \\
\tb ( \lambda )& = & \lambda \min \{ \beta, \max \{ \alpha- \gamma, 0 \} \} + ( 1 - \lambda ) \min \{ \alpha, \beta \} \\
\end{array}
\end{equation*}
Notice also that, no matter how $\lambda \in [0,1]$ is selected, the following quantity is constant:
\[ \ta ( \lambda ) + \tb ( \lambda ) = \min \{ \beta + \gamma , \alpha \} \]
So, it is natural to wonder where the minimum for $\lambda_{\max} ( \zeta_{T_d} ( \ta ( \lambda ), \tb( \lambda ) ) )$ is achieved.
We establish now that the optimal choice is the one presented in Theorem \ref{th2}.
This is indeed natural, as it amounts to maximize in \eqref{ineqD} the contribution of the diameter $\Delta_k(t)$ of the smallest set $\cN_k$.

It is well known that, in general, the spectral radius of a linear combination of matrices needs {\em not\/} be a convex function; see e.g.\ \cite[p.\ 364]{HOR94}.
The additional structure provided by this problem, however allows to exactly identify the optimal value of $\lambda$ as explained below.
Consider the cone
\begin{equation}
\label{cone}
K_{T_d} \doteq \{ w : 0\leq w_1 \leq w_2 \leq w_3 \leq \ldots \leq w_{T_d} \}\ ,
\end{equation}
which is contained in the nonnegative orthant.
It is straightforward to verify that $ \zeta_{T_d} ( \ta, \tb ) K_{T_d} \subseteq K_{T_d}$. In particular, then, by Perron-Frobenius theory \cite[Theorem 3.1]{VAN}, we know that $K_{T_d}$ contains an eigenvector, say $\hat{w}$, relative to the dominant eigenvalue of $\zeta_{T_d} ( \ta, \tb )$.

Moreover application of a classical characterization of the spectral radius $\lambda_{\max} ( \zeta_{T_d} ( \ta,\tb) )$ \cite[Theorem 4]{RHE}, see also \cite{FOR}, yields:
\[
\lambda_{\max} ( \zeta_{T_d} ( \ta,\tb) )
= \max_{w \geq 0}\ \min\left\{
 \frac{ [ \zeta_{T_d} ( \ta, \tb ) w ]_i }{w_i}\ :\ i\in\{1,\dots, T_d\},\ w_i \neq 0
 \right\}\ .
 \]
The previous formula is established considering the nonnegative orthant ($w\geq 0$).

Now, the matrix $ \zeta_{T_d} ( \ta,\tb)$ is irreducible for the cone constituted by the nonnegative orthant (that is: irreducible in the usual sense).
One deduces that it admits exactly {\em one\/} eigenvector in the interior of the positive orthant \cite[Theorem 4.2]{VAN}.
The latter has to be $\hat{w}$, and we thus have

\begin{eqnarray*}
\min\left\{
 \frac{ [ \zeta_{T_d} ( \ta, \tb ) \hat{w} ]_i }{\hat{w}_i}\ :\ i\in\{1,\dots, T_d\}
 \right\}
& = &
\max_{w \geq 0}\ \min\left\{
 \frac{ [ \zeta_{T_d} ( \ta, \tb ) w ]_i }{w_i}\ :\ i\in\{1,\dots, T_d\},\ w_i \neq 0
 \right\}\\
 \\
 & \geq &
 \max_{w \in K_{T_d}}\min\left\{
 \frac{ [ \zeta_{T_d} ( \ta, \tb ) w ]_i }{w_i}\ :\ i\in\{1,\dots, T_d\},\ w_i \neq 0
 \right\}\\
 & \geq &
 \min\left\{
 \frac{ [ \zeta_{T_d} ( \ta, \tb ) \hat{w} ]_i }{\hat{w}_i}\ :\ i\in\{1,\dots, T_d\}
 \right\} \ ,
\end{eqnarray*}
using the fact that $\hat{w}$ is also an element of $K_{T_d}$.
In the previous series of inequalities, $w\geq 0$ refers to the componentwise order relation, and the fact that $K_{T_d}$ is included in the nonnegative orthant has been used.
We thus deduce:

\begin{equation}
\label{eigenvalue} 
\lambda_{\max} ( \zeta_{T_d} ( \ta,\tb) )
=  \max_{w \in K_{T_d}}\min\left\{
 \frac{ [ \zeta_{T_d} ( \ta, \tb ) w ]_i }{w_i}\ :\ i\in\{1,\dots, T_d\},\ w_i \neq 0
 \right\} \ .
\end{equation}

Now, the following expression holds for $ \zeta_{T_d} ( \ta ( \lambda ), \tb ( \lambda ) )$,
\begin{equation*}
\label{monotonicity}
 \zeta_{T_d} ( \ta ( \lambda ), \tb ( \lambda ) ) =  \min \{ \beta + \gamma , \alpha \} \bfo e_{T_d}\t  + \tb ( \lambda ) I_{T_d} + \ta ( \lambda ) J_{T_d}\ .
 \end{equation*} 
Hence, for all $w \in K_{T_d}$ the function $ \zeta_{T_d} ( \ta ( \lambda ), \tb ( \lambda ) ) w$ is (entrywise) monotone with respect to $\lambda$.
Henceforth, by \eqref{eigenvalue},
$\lambda_{\max}(\zeta_{T_d} ( \ta ( \lambda ), \tb ( \lambda ) ))$ is also monotone with respect to $\lambda$.
In particular, the minimum is achieved when $\ta$ is maximum. In our notation this corresponds to $\lambda = 1$.
In conclusion, the optimal estimate is the one obtained in Theorem \ref{th2}.
\hfill$\square$
\end{rema}

\subsection{Proof of Theorem \protect\ref{th3}}
\label{pr_th3}

\noindent $\bullet$
Let the characteristic polynomial be $\chi_T(s)\doteq\det (sI_T-(1-\ta-\tb)\ \bfo e_T\t -\tb I_T-\ta J_T)$.
The matrix $(1-\ta-\tb)\ \bfo e_T\t +\tb I_T+\ta J_T$ being irreducible and nonnegative, its spectral radius is indeed a (real) eigenvalue; see \cite[Theorem 8.4.4]{HOR90}.
We thus aim to estimate the largest zero of $\chi_T(s)$, that we denote $\rho_T(\ta,\tb)$.

Developing the determinant with respect to its first column yields, for any $T\geq 2$,
\begin{equation}
\label{det1}
\chi_T(s) = (s-\tb)\chi_{T-1}(s)-{\ta}^{T-1}(1-\ta-\tb)\ .
\end{equation}

It is then easy to check that
\begin{equation*}
\label{det2}
\chi_T(s)
=
(s-\tb)^T+ (\ta+\tb-1)\left(
(s-\tb)^{T-1}+\ta(s-\tb)^{T-2}+\dots+{\ta}^{T-2}(s-\tb)+{\ta}^{T-1}
\right)\ ,
\end{equation*}
and that $\chi_T(s)=0$ if and only if \eqref{real} holds.
In this relation, the coefficient $\frac{1-\tb}{\ta}$ is at most equal to 1.
One sees directly that the scalar $\rho_T(\ta,\tb)$, the largest root of \eqref{real}, is thus smaller than 1.

\noindent $\bullet$
The fact that $\chi_1(s)=(s-1+\ta)$ and formula \eqref{det1} permit to show recursively that, for any $T\geq 2$,
\[
\chi_T(1-\ta),\ \chi_T(\tb)\leq 0\ .
\]
As
\[
\lim_{s\to +\infty}\chi_T(s) =+\infty
\]
(in which the sign of the limit is crucial), one then deduces that, for any $T\geq 1$,
\[
\rho_T(\ta,\tb)\geq 1-\ta,\tb\ .
\]

\comment{
`Completing' $(1-\ta-\tb)\ \bfo e_1\t +\tb I_T+\ta J_T$ to obtain a $(T+1)\times (T+1)$ matrix by putting $(1-\ta-\tb)$ in the $(T+1, 1)$-component; $\tb$ in the $(T+1,T+1)$-component; and $\ta$ in the $(T,T+1)$-component yields a {\em stochastic\/} matrix with second largest eigenvalue smaller than 1.
Thus, we have, for all $T\geq 1$
\[
\rho_T(\ta,\tb)< 1\ .
\]
}

\noindent $\bullet$
One deduces from \eqref{det1} that
\[
\chi_{T-1}(s)=0\Rightarrow\chi_T(s)=-{\ta}^{T-1}(1-\ta-\tb)\leq 0\ ,
\]
from which we conclude that, for $T\geq 2$,
\[
\rho_T(\ta,\tb)\geq \rho_{T-1}(\ta,\tb)\ .
\]

\noindent $\bullet$
The fact that $\rho_T(\ta,\tb)\leq\ta+\tb$ if and only if $T\leq\frac{\ta}{1-\ta-\tb}$ is obtained from the observation that
\[
\frac{1}{\alpha^{\star T}}\chi_T(\ta+\tb)
=
(T+1)-\frac{1-\tb}{\ta} T = \frac{1}{\ta}((\ta+\tb-1)T +\ta)\ .
\]

\noindent $\bullet$
Let us now determine the limit of $\rho_T(\ta,\tb)$ when $T\to +\infty$.
As asserted by the inequality just proved, one may assume, for large enough $T$, that $\rho_T(\ta,\tb)>\ta+\tb$.
In these conditions, one may multiply by $\left(
\frac{\rho_T(\ta,\tb)-\tb}{\ta}-1
\right)$ both sides of \eqref{real} taken in $s=\rho_T(\ta,\tb)$.
This yields the identity:
\[
\left(
\frac{\rho_T(\ta,\tb)-\tb}{\ta}
\right)^{T+1}-1
=
\left(
\frac{1-\tb}{\ta}
\right)
\left(\left(
\frac{\rho_T(\ta,\tb)-\tb}{\ta}
\right)^T-1
\right)\ ,
\]
that is:
\begin{equation}
\label{iden}
(\rho_T(\ta,\tb)-1)
\left(
\frac{\rho_T(\ta,\tb)-\tb}{\ta}
\right)^T
=
(\ta+\tb-1)>0
\ .
\end{equation}

The results in Theorem \ref{th3} show that the factor $(\rho_T(\ta,\tb)-1)$ is bounded when $T\to +\infty$.
On the other hand, the term $\left(
\frac{\rho_T(\ta,\tb)-\tb}{\ta}
\right)^T$ either tends to zero, or to infinity.
In these circumstances, in order for the previous identity to be verified for $T$ sufficiently large, the only possibility is that, when $T\to +\infty$,
\[
(\rho_T(\ta,\tb)-1)\to 0,\quad
\left(
\frac{\rho_T(\ta,\tb)-\tb}{\ta}
\right)^T\to +\infty\ .
\]
Considering then that
\[
\left(
\frac{\rho_T(\ta,\tb)-\tb}{\ta}
\right)^T
\sim
\left(
\frac{1-\tb}{\ta}
\right)^T\ ,
\]
identity \eqref{iden} yields the announced expansion formula.

\subsection{Proof of Theorem \protect\ref{th4}}
\label{pr_th4}

\noindent $\bullet$
As exposed in Remark \ref{convexfreedom2}, $\rho_T(\ta,\tb)$ is the largest eigenvalue of $\zeta_T(\ta,\tb)$, associated to a single eigenvector in $\Rset^T$, say $\hat{w}$, contained in the cone $K_T$ defined in \eqref{cone}.

By the very definition of this matrix,
\[
\zeta_T(\ta,\tb)\hat{w} = \rho_T(\ta,\tb) \hat{w}\ .
\]
For any other pair $(\alpha^{\star\star},\beta^{\star \star})$ and for any $i\in\{2,\dots, T\}$, one has
\[
[\zeta_T(\alpha^{\star\star},\beta^{\star \star})\hat{w}]_i
= (1-\alpha^{\star\star}-\beta^{\star \star})[\hat{w}]_T
+\beta^{\star\star}[\hat{w}]_i
+\alpha^{\star\star}[\hat{w}]_{i-1}\ .
\]
Thus,
\begin{eqnarray*}
[\zeta_T(\alpha^{\star\star},\beta^{\star \star})\hat{w}]_i
-[\zeta_T(\ta,\tb)\hat{w}]_i
& = &
(\ta+\tb-\alpha^{\star\star}-\beta^{\star \star})[\hat{w}]_T
+(\beta^{\star\star}-\tb)[\hat{w}]_i
+(\alpha^{\star\star}-\ta)[\hat{w}]_{i-1}\\
& = &
(\beta^{\star\star}-\tb)([\hat{w}]_i-[\hat{w}]_T)
+(\alpha^{\star\star}-\ta)([\hat{w}]_{i-1}-[\hat{w}]_T)\ .
\end{eqnarray*}
Taken into account the fact that $\hat{w} \in K_T$, which implicates $[\hat{w}]_i-[\hat{w}]_T, [\hat{w}]_{i-1}-[\hat{w}]_T\leq 0$, one deduces that: if  $(\alpha^{\star\star},\beta^{\star \star})\leq(\ta,\tb)$, then, for $i=2,\dots, T_d$,
\[
[\zeta_T(\alpha^{\star\star},\beta^{\star \star})\hat{w}]_i
\geq [\zeta_T(\ta,\tb)\hat{w}]_i
= \rho_T(\ta,\tb)[\hat{w}]_i\ .
\]

Similarly, for $i=1$, one has
\[
[\zeta_T(\alpha^{\star\star},\beta^{\star \star})\hat{w}]_1
= (1-\alpha^{\star\star}-\beta^{\star \star})[\hat{w}]_T
+\beta^{\star\star}[\hat{w}]_1\ ,
\]
and
\begin{eqnarray*}
[\zeta_T(\alpha^{\star\star},\beta^{\star \star})\hat{w}]_1
-[\zeta_T(\ta,\tb)\hat{w}]_1
& = &
(\ta+\tb-\alpha^{\star\star}-\beta^{\star \star})[\hat{w}]_T
+(\beta^{\star\star}-\tb)[\hat{w}]_1\\
& = &
(\beta^{\star\star}-\tb)([\hat{w}]_1-[\hat{w}]_T)
-(\alpha^{\star\star}-\ta)[\hat{w}]_T\ ,
\end{eqnarray*}
so
\[
[\zeta_T(\alpha^{\star\star},\beta^{\star \star})\hat{w}]_1
\geq [\zeta_T(\ta,\tb)\hat{w}]_1
= \rho_T(\ta,\tb)[\hat{w}]_1\ .
\]

To summarize, we thus have:
\[
\zeta_T(\alpha^{\star\star},\beta^{\star \star})\hat{w} \geq \rho_T(\ta,\tb) \hat{w}\ ,
\]
and the fact that $\zeta_T(\alpha^{\star\star},\beta^{\star \star})$ has nonnegative components certainly implies \cite[Theorem 4]{RHE} that
\[
\rho_T(\ta,\tb) \leq \lambda_{\max}(\zeta_T(\alpha^{\star\star},\beta^{\star \star}))
=\rho_T(\alpha^{\star\star},\beta^{\star \star})\ .
\]
The nonincreasingness of $\rho_T(\ta,\tb)$ with respect to $(\ta,\tb)$ is thus proved.

\noindent $\bullet$
Careful examination of \eqref{tatb} shows that
\[
\frac{\partial\ta}{\partial\alpha}, \frac{\partial\ta}{\partial\gamma}, \frac{\partial\tb}{\partial\alpha}, \frac{\partial\tb}{\partial\beta}, -\frac{\partial\tb}{\partial\gamma}\in\{0,1\},\quad
\frac{\partial\ta}{\partial\beta}=0\ .
\]

One thus shows easily, with the help of the previous result, that $\rho_T(\alpha,\beta,\gamma)$ is nonincreasing in $\alpha$ and in $\beta$.

The study of the influence of $\gamma$ requires to distinguish between three different cases.
Based on \eqref{tatb}, one has:

\begin{itemize}
\item if $\alpha\leq\gamma$, then $\ta=\alpha$, $\tb=\beta$, so that $\rho_T(\alpha,\beta,\gamma)$ is locally insensitive to variations of $\gamma$;

\item if $\alpha\geq\beta+\gamma$, then $\ta=\gamma$, $\tb=\beta$; so that an increase of $\gamma$ induces an increase of $\ta$, does not affect $\tb$, and therefore provokes a decrease of $\rho_T(\alpha,\beta,\gamma)$.

\item in the intermediate case where $\gamma\leq\alpha\leq\beta+\gamma$, one has $\ta=\gamma$, $\tb=\alpha-\gamma$, so that the rationale developed in Remark \ref{convexfreedom2} holds, indicating a decrease of $\rho_T(\alpha,\beta,\gamma)$ with respect to $\ta$, i.e.\ with respect to $\gamma$.
\end{itemize}

\noindent $\bullet$
We now demonstrate the comparison relation stated in Theorem \ref{th4} for constant value of $\beta+\gamma$.
When this quantity, together with $\alpha$, is unchanged, the sum $\ta+\tb$ is conserved.
Arguing one more time as in Remark \ref{convexfreedom2}, $\rho_T(\ta,\tb)$ is nonincreasing with respect to $\ta=\min\{\alpha,\gamma\}$, and thus nondecreasing with respect to $\beta$.

\noindent $\bullet$
Concerning the computation of $\rho_T(\alpha,\beta,\gamma)$ when $\alpha$ or $\gamma$ is null, it is straightforward when remarking that in both cases, $\ta=0$.

Conversely, due to the monotonicity property already established, the largest value of $\rho_T(\alpha,\beta,\gamma)$ is attained for extremal values.
The nullity of $\beta$ alone, which yields $\ta=\min\{\alpha,\gamma\}$, $\tb=0$, does not induce that the maximal value is attained (see \eqref{real}).
Also, when $\alpha=0$ but $\gamma\neq 0$ (resp.\ $\gamma=0$ but $\alpha\neq 0$), then $\frac{\partial\ta}{\partial\alpha}=1$ (resp.\ $\frac{\partial\ta}{\partial\gamma}=1$), and any increase of $\alpha$ (resp.\ $\gamma$) produces a decrease of $\rho_T(\alpha,\beta,\gamma)$.

\noindent $\bullet$
Last, when $\alpha=1$ and $\beta+\gamma=1$, then $\ta=1-\beta$, $\tb=\beta$, and the matrix $\zeta_T(\ta,\tb)$ is upper triangular, with $\beta$ on the diagonal.

Reversely, decrease of $\alpha$ or $\beta$ (resp.\ $\gamma$) implies increase with respect to $1-\gamma$ (resp.\ $\beta$).

\end{document}